\documentclass[11pt,oneside]{amsart}
\usepackage[margin=1in]{geometry}

\usepackage{amssymb, amsmath}
\usepackage{float}
\usepackage{pdfpages}
\usepackage{epstopdf}
\usepackage{comment}
\usepackage{booktabs}
\usepackage{graphicx}
\usepackage{color}
\usepackage{pinlabel}
\usepackage{mathtools,xparse}
\usepackage{tikz}
\usetikzlibrary{arrows,positioning,shapes,fit,calc}
\usepackage{enumerate}

\usepackage[colorlinks = true,
            linkcolor = red,
            urlcolor  = red,
            citecolor = red,
            anchorcolor = red]{hyperref}
\usepackage{amsrefs}
\usepackage[pagewise]{lineno}

\usepackage{pinlabel}

\theoremstyle{plain}
\newtheorem{theorem}{Theorem}[section]

\newtheorem*{theorem*}{Theorem}
\newtheorem*{maintheorem-intro}{Theorem}
\newtheorem*{maintheorem-intro-2}{Theorem~\ref{Bridge number and genus}}
\newtheorem*{theorem-cablingconj}{Theorem~\ref{apps1} (1)}
\newtheorem*{theorem-toroidal}{Specialization of Theorem~\ref{apps1} (2)}
\newtheorem*{theorem-lens}{Theorem~\ref{Bounding distance - special}(1)}
\newtheorem*{theorem-SFS}{Theorem~\ref{Bounding distance - special}(2)}
\newtheorem*{theorem-cosmetic}{Theorem}
\newtheorem*{theorem-bridge}{Specialization of Corollary~\ref{Cor: exceptional bridge}}
\newtheorem*{theorem-Heeggenus}{Corollary~\ref{Cor: exceptional bridge} (2)}

\newtheorem{lemma}[theorem]{Lemma}

\theoremstyle{definition}
\newtheorem{remark}[theorem]{Remark}

\theoremstyle{definition}

\newcommand{\bi}{\begin{itemize}}
\newcommand{\ei}{\end{itemize}}
\newcommand{\be}{\begin{enumerate}}
\newcommand{\ee}{\end{enumerate}}

 \usepackage[abs]{overpic}
\begin{document}
   \title[On the Triple Point Number of Surface-Links in Yoshikawa's Table]{On the Triple Point Number of Surface-Links in Yoshikawa's Table}
   \author{Nicholas Cazet}
   
   \begin{abstract}

Yoshikawa made a table of knotted surfaces in $\mathbb{R}^4$ with ch-index 10 or less. This remarkable table is the first to enumerate knotted surfaces analogous to the classical prime knot table. A broken sheet diagram of a surface-link is a generic projection of the surface in $\mathbb{R}^3$ with crossing information along its singular set. The minimal number of triple points among all broken sheet diagrams representing a given surface-knot is its triple point number. This paper compiles the known triple point numbers of the surface-links represented in Yoshikawa's table and calculates or provides bounds on the triple point number of the remaining surface-links.

 \end{abstract}

\maketitle

\section{Introduction}

A {\it surface-link} is a smoothly embedded closed surface  in $\mathbb{R}^4$.  Two surface-links are {\it equivalent} if they are related by a smooth ambient isotopy. A surface-link diffeomorphic to a 2-sphere is called a {\it 2-knot}, and a surface-link with only 2-sphere components each bounding a  3-ball disjoint from the other is called a {\it trivial 2-link}. A surface-link $F$ is called {\it ribbon} if there is a trivial 2-link $\mathcal{O}$ and a collection of embedded (3-dimensional) 1-handles attaching to $\mathcal{O}$ such that $F$ is the result of surgery along these 1-handles. These 1-handles intersect $\mathcal{O}$ only in their attaching regions.

For an orthogonal projection $p: \mathbb{R}^4\to\mathbb{R}^3$, a surface-link $F$ can be perturbed slightly so that $p(F)$ is a generic surface.  Each point of the generic surface $p(F)$ has a neighborhood in 3-space diffeomorphic  to $\mathbb{R}^3$ such that the image of the  generic surface under the diffeomorphism looks like 1, 2, or 3 coordinate planes or the cone on a figure 8 (Whitney umbrella).
These points are called regular points, double points,  triple points, and  branch points. Triple points and branch points are isolated while double points are not and lie on curves called {\it double point curves}. The union of non-regular points is the {\it singular set} of the generic surface \cite{CKS}, \cite{kamada2017surface}.

 \begin{figure}[h]
\begin{overpic}[unit=.5mm,scale=.25]{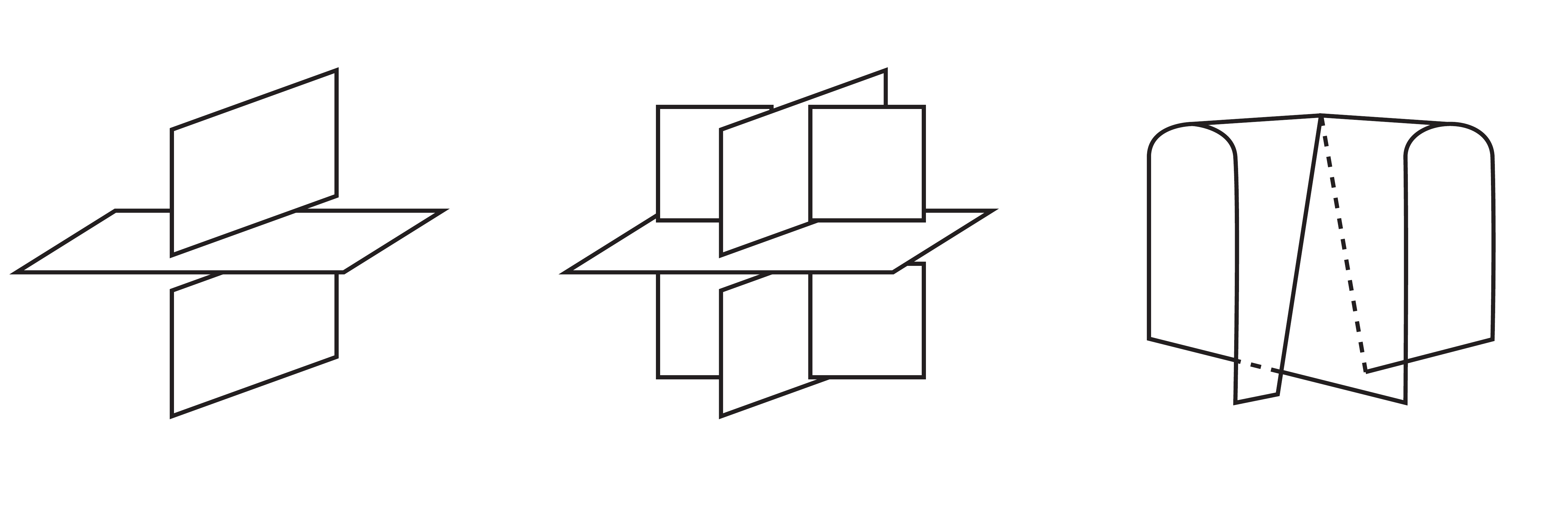}

\end{overpic}
\caption{Local images of a double point, triple point, and branch point.}
\label{fig:1}
\end{figure}

A {\it broken sheet diagram} of a surface-link $F$ is a generic projection $p(F)$ with  consistently broken sheets along its singular set,  see Figure \ref{fig:1} and \cite{cartersaito}. The sheet that lifts below the other, in respect to a height function determined by the direction of the orthogonal projection $p$, is locally broken at their intersection. All surface-links admit a broken sheet diagram, and all broken sheet diagrams lift to surface-links in 4-space. Although, not all compact, generic surfaces in 3-space can be given a broken sheet structure  \cite{Carter1998}. The minimal number of triple points among all broken sheet diagrams representing a surface-link $F$ is called the {\it triple point number} of $F$ and is denoted $t(F)$. A surface-link with a triple point number of 0 is called {\it pseudo-ribbon}. Every ribbon surface-link is  pseudo-ribbon, and every pseudo-ribbon 2-knot is ribbon \cite{yajima1964}. However, the turned spun torus surface-link of a non-trivial classical knot is pseudo-ribbon and non-ribbon \cite{boyle}, \cite{iwase}, \cite{livingston}.

A {\it singular link diagram} is an immersed link diagram in the plane with crossings and traverse double points called {\it vertices}. At each vertex assign a {\it marker}, a local choice of two non-adjacent regions in the complement of the vertex. Such a marked singular link diagram is called a {\it ch-diagram} \cite{CKS}, \cite{kamada2017surface}, \cite{kamadakim},   \cite{Yoshikawa}. One of the two smoothings of a vertex connects the two regions of the marker, the positive resolution $L^+$, and one separates the marker's regions, the negative resolution $L^-$, see Figure \ref{fig:smoothing}.   If $L^-$ and $L^+$ are unlinks, then the ch-diagram is said to be {\it admissible}. Admissible ch-diagrams represent surface-links and induce broken sheet diagrams, this is described in Section \ref{sec2}, and every surface-link defines an admissible ch-diagram.

%For any self-indexing Morse function $f$ on the surface-link generically project the singular link $f^{-1}(1)$ onto a plane. Markers are placed at the vertices based on the opening of the saddles from $f^{-1}(1-\epsilon)$ to $f^{-1}(1+\epsilon)$.  

\begin{figure}[h]

\begin{overpic}[unit=.46mm,scale=.6]{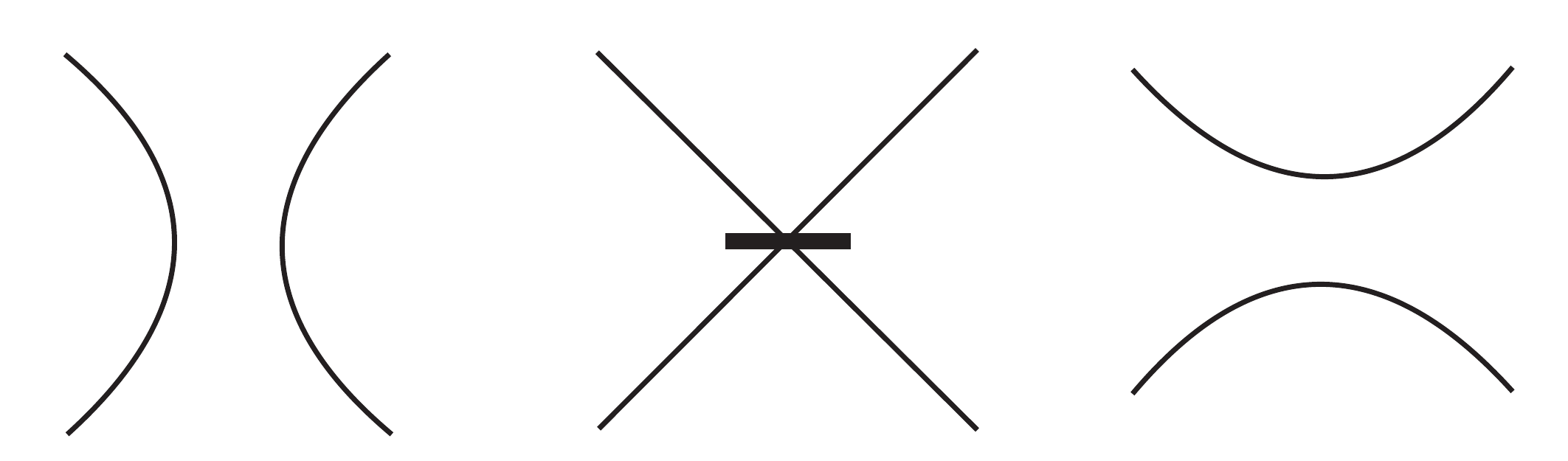}\put(37,-2){$L^-$}\put(233,-2){$L^+$}

\end{overpic}

\caption{Smoothings of a marked vertex.}
\label{fig:smoothing}
\end{figure}

Including the 3 Reidemeister moves of classical link diagrams, there are 8 moves on ch-diagrams called {\it Yoshikawa moves}, see p.61 of \cite{kamada2017surface}. Two admissible ch-diagrams represent equivalent surface-links if and only if they are Yoshikawa move equivalent, Theorem 3.6.3 of \cite{kamada2017surface}. The ch-index of a ch-diagram is the number of marked vertices plus the number of crossings. The {\it ch-index} of a surface-link $F$, denoted ch($F$), is the minimum of ch-indices ch$(D)$ among all admissible ch-diagrams $D$ representing $F$. A sufrace-link $F$ is said to be {\it weakly prime} if $F$ is not the connected sum of any two surfaces $F_1$ and $F_2$ such that ch$(F_i)$ $<$ ch$(F)$. Yoshikawa classified weakly prime surface-links whose ch-index is 10 or less in \cite{Yoshikawa}. He generated a table of their representative ch-diagrams up to orientation and mirror. A rendition of his table is done in Figure \ref{fig:yoshikawa}. His notation is of the form $I_k^{g_1,\dots,g_c}$ where $I$ is the surface-link's ch-index and $g_1,\dots,g_c$ are the genera of its components.

 This paper considers the triple point number of the surface-links represented in Yoshikawa's table. Section \ref{sec3} tabulates the known triple point numbers of the surface-links from his table and calculates or provides bounds on the triple point number of the remaining surface-links.

\begin{figure}[h]

\begin{overpic}[unit=.5mm,scale=.6]{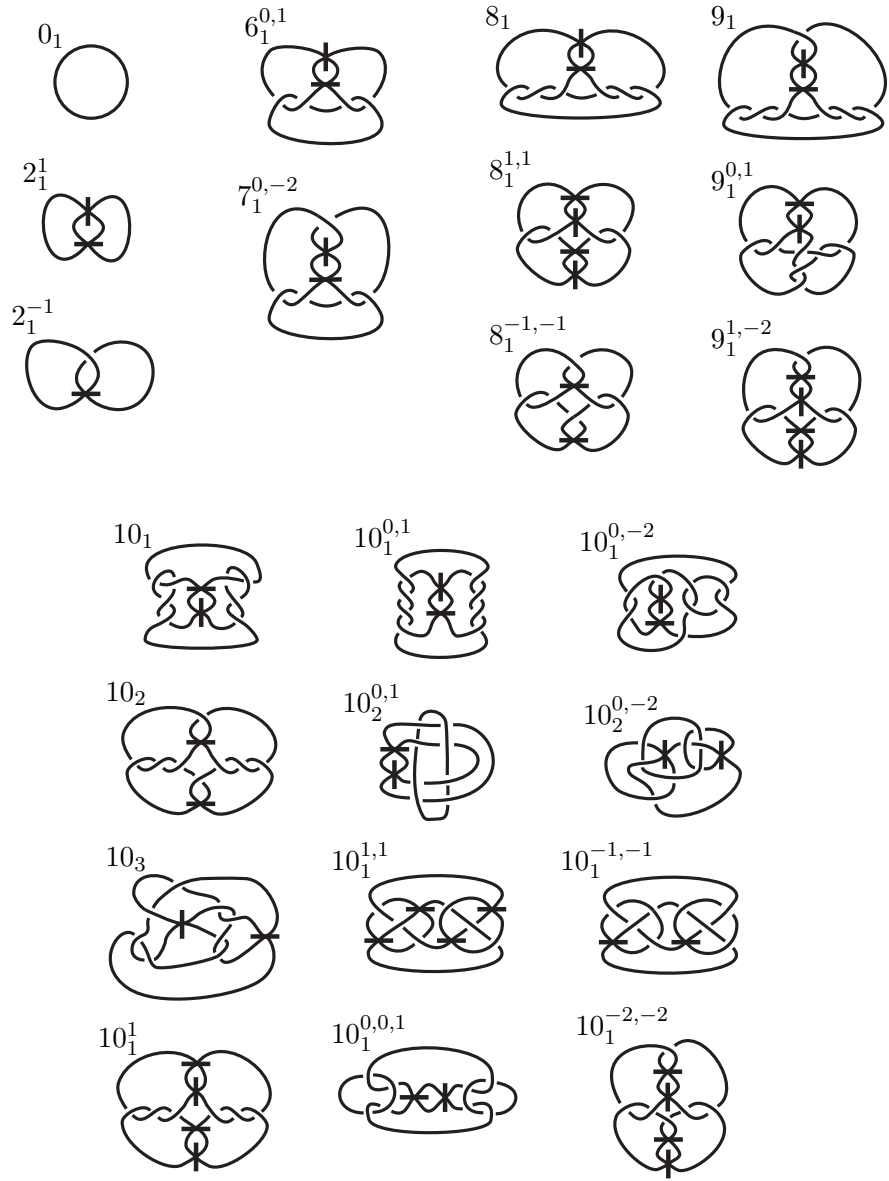}\put(22,308){$0_1$}\put(18,271){$2_1^1$}\put(15,233){$2_1^{-1}$}\put(77,310){$6_1^{0,1}$}\put(75,265){$7_1^{0,-2}$}
\put(141,313){$8_1$}\put(142,272){$8_1^{1,1}$}\put(142,228){$8_1^{-1,-1}$}\put(201,313){$9_1$}\put(201,269){$9_1^{0,1}$}\put(201,227){$9_1^{1,-2}$}

\put(42,175){$10_1$}\put(40,133){$10_2$} \put(40,89){$10_3$}\put(38,41){$10_1^1$}\put(106,174){$10_1^{0,1}$}\put(103,130){$10_2^{0,1}$}\put(100,88){$10_1^{1,1}$}\put(100,42){$10_1^{0,0,1}$} \put(166,173){$10_1^{0,-2}$}\put(167,127){$10_2^{0,-2}$} 
\put(161,88){$10_1^{-1,-1}$} \put(165,44){$10_1^{-2,-2}$}

\end{overpic}

\caption{Yoshikawa's table of surface-links with ch-index no greater than 10 \cite{Yoshikawa}.}
\label{fig:yoshikawa}
\end{figure}

\section{Induced Broken Sheet Diagrams of Admissible Ch-diagrams}
\label{sec2}

Given a surface-knot $F\subset \mathbb{R}^4$ and a vector ${\bf v}\in \mathbb{R}^4$, perturb $F$ such that the orthogonal projection of $\mathbb{R}^4$ onto $\mathbb{R}$ in the direction of ${\bf v}$ is a Morse function. For any $t\in \mathbb{R}$, let $\mathbb{R}_t^3$ denote the affine hyperplane orthogonal to ${\bf v}$ that contains the point $t {\bf v}$.  Morse theory allows for the assumption that all but finitely many of the non-empty cross-sections $F_t=\mathbb{R}_t^3 \cap F$ are classical links. The decomposition $\{F_t\}_{t\in\mathbb{R}}$ is called a {\it motion picture} of $F$. It may also be assumed that the exceptional cross-sections contain minimal points, maximal points, and/or singular links \cite{CKS}, \cite{kamada2017surface}. 

There is a product structure between Morse critical points implying that only finitely many cross-sections are needed to decompose, or construct, $F$. Although, a sole cross-section of a product region does not uniquely determine its knotting, i.e. its ambient isotopy class relative boundary, see \cite{CKS}. Project the cross-sections $\{F_t\}_{t\in\mathbb{R}}$  onto a plane to get an ordered family of planar diagrams containing classical link diagrams, minimal points, maximal points, and singular link diagrams. These planar diagrams are {\it stills} of the motion picture. The collection of all stills is also referred to as a motion picture of the surface-link. The product structure between critical points implies that cross-sections between consecutive critical points represent the same link. Therefore, there is a sequence of Reidemeister moves and planar isotopies between the stills of a motion picture that exists between consecutive critical points. There is a translation of Reidemeister moves in a motion picture to sheets in a broken sheet diagram.

  Associate the time parameter of a Reidemeister move with the height of a local broken sheet diagram.  A Reidemeister III move gives a triple point diagram, a Reidemeister I move corresponds to a branch point, and a Reidemeister II move corresponds to a maximum or minimum of a double point curve, see Figure \ref{fig:r}. Triple points of the induced broken sheet diagram are in correspondence with the Reidemeister III moves in the motion picture. To generate a broken sheet diagram of the entire surface-link include sheets containing saddles for each saddle point in the motion picture.

 \begin{figure}[h]

\begin{overpic}[unit=.5mm,scale=.6]{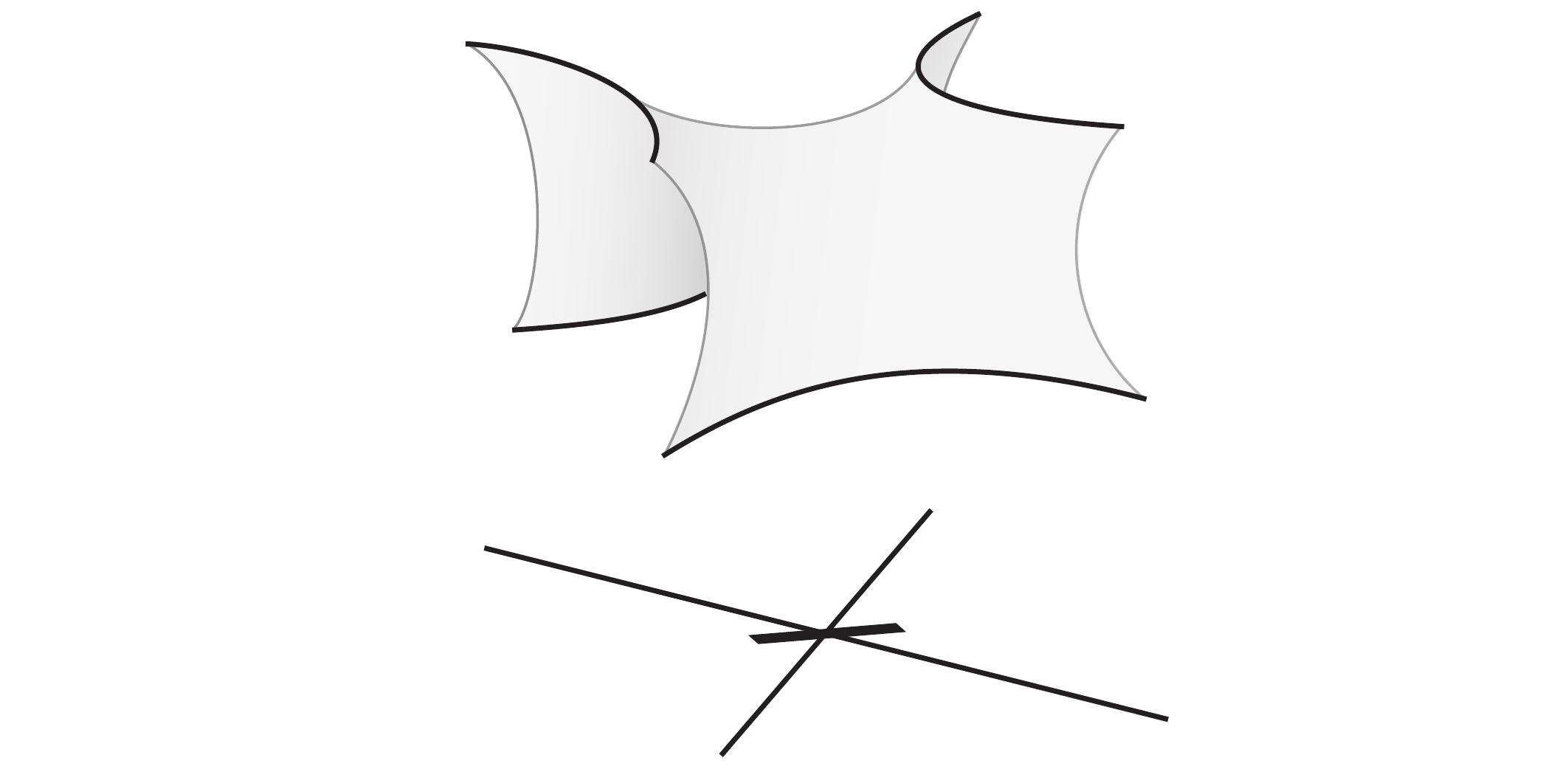}

\end{overpic}

\caption{Induced saddle sheet.}
\label{fig:saddle}
\end{figure}

Now, consider an admissible ch-diagram. There is a finite sequence of Reidemeister moves that takes $L^-$ and $L^+$ to crossing-less diagrams $O^-$ and $O^+$. Translate these Reidemeister moves to a broken sheet diagram.  In between the still of $L^-$ and $L^+$ and for each marked vertex include sheets containing the saddles traced by transitioning from $L^-$ to  $L^+$ in the local picture of Figure \ref{fig:smoothing}. These saddles are locally pictured in Figure \ref{fig:saddle}. Finally,  cap-off $O^-$ and $O^+$ with trivial disks to produce a broken sheet diagram of a surface-link.

 \section{The Triple Point Numbers of Surface-Links Represented in Yoshikawa's Table}
 \label{sec3}

The three trivial surface-knots $0_1$, $2_1^1$, and $2_1^{-1}$ are clearly pseudo-ribbon. Most of the remaining surface-links in Yoshikawa's table are ribbon.

\begin{figure}[h]
\includegraphics[scale=.3]{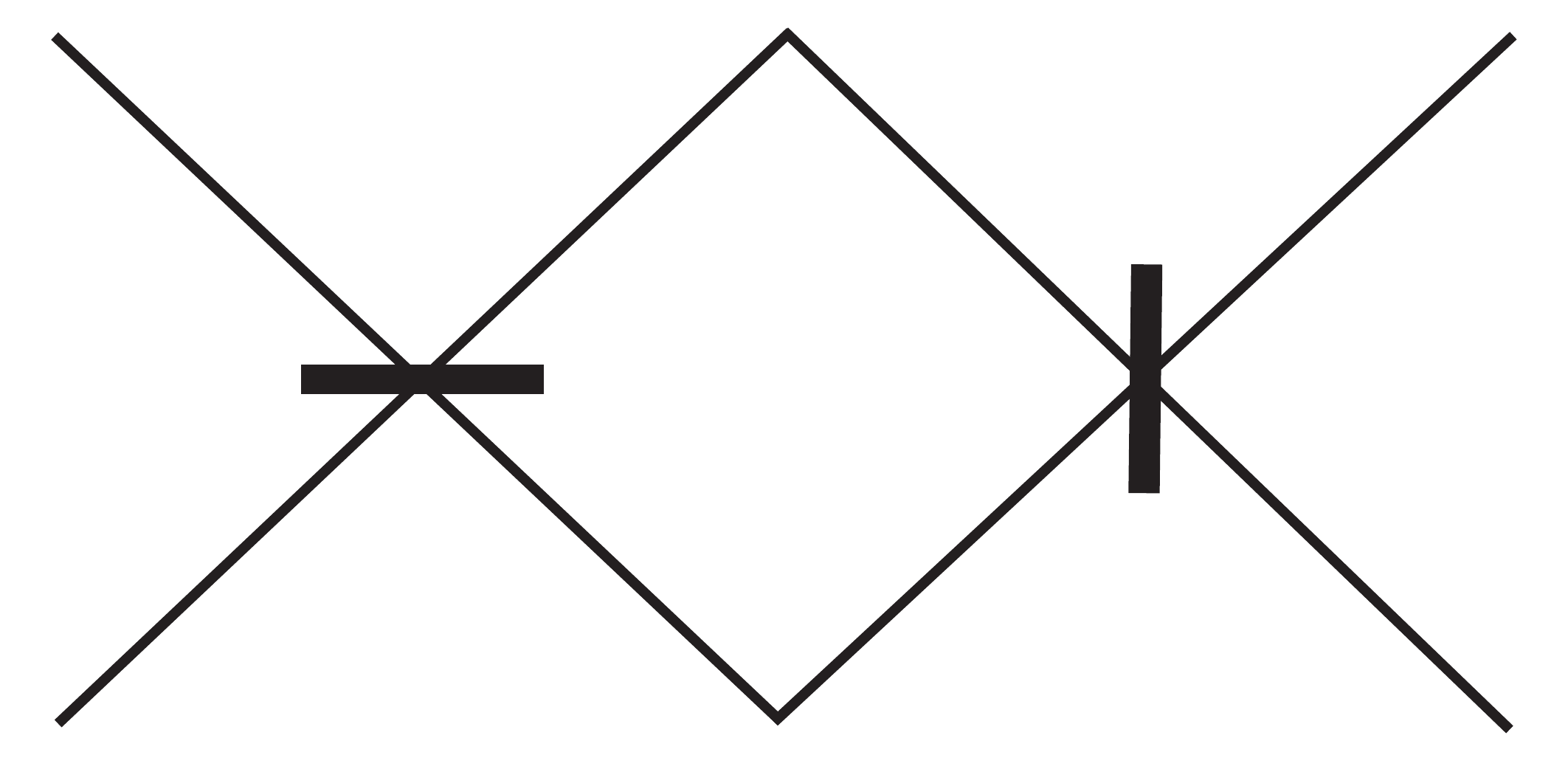}

\caption{Marked vertex pattern.}
\label{fig:ribbon}
\end{figure}

\begin{figure}[h]

\begin{overpic}[unit=.5mm,scale=.6]{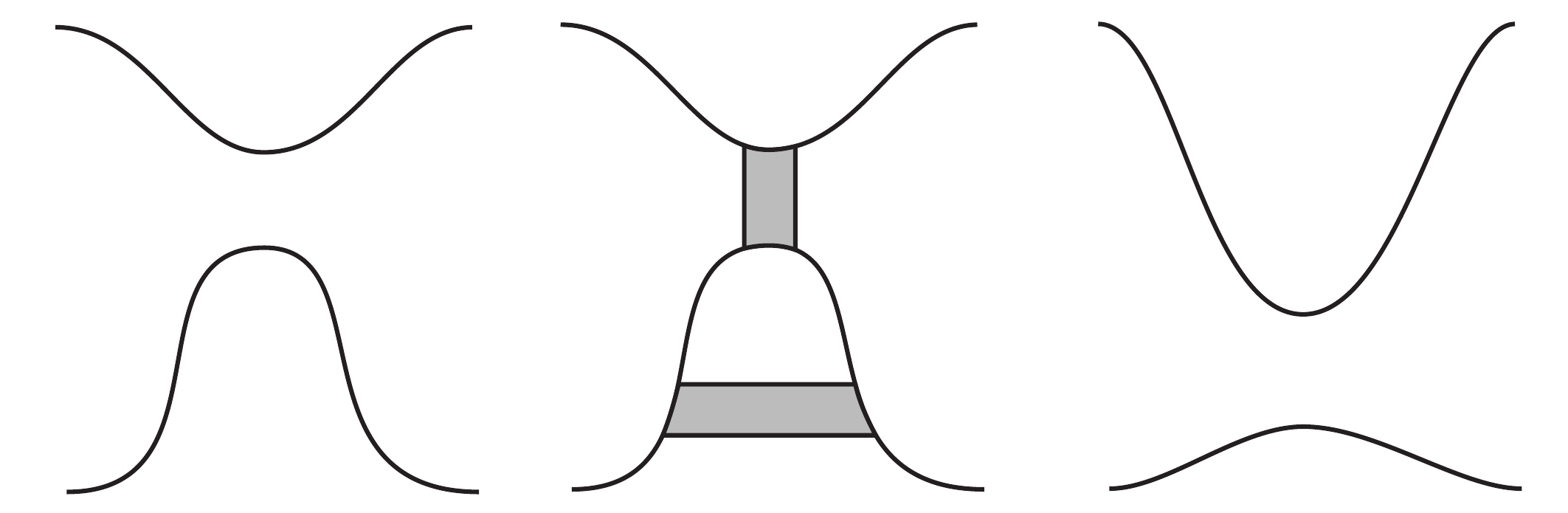}

\end{overpic}

\caption{Motion picture of the saddles represented by Figure \ref{fig:ribbon}.}
\label{fig:band}
\end{figure}

\begin{theorem}

Admissible ch-diagrams with vertices paired and isolated in the pattern of Figure \ref{fig:ribbon} represent ribbon surface-links.

\label{thm:ribbon}
\end{theorem} 

\begin{proof}
Figure \ref{fig:band} shows that the unlink diagrams of the positive and negative resolutions are identical, up to a small planar isotopy. The positive and negative resolutions are the boundary of trivial disk systems, families of disjoint 2-disks with one maximal (or minimal) point. Between these trivial disk systems include a product region representing the aforementioned small planar isotopies between the two resolutions to produce a collection of 2-knots each with just one minimal point and one maximal point.  With no saddles, the constructed surface-link with all 2-sphere components is a trivial 2-link since any two trivial disk systems that have the same boundary are smoothly ambient isotopic (rel boundary) \cite{kamada2017surface}.   Figure \ref{fig:handle} shows that the addition of the saddles represented by the marked vertices is equivalent to 1-handle surgery on this trivial 2-link.

\end{proof}

\begin{figure}
\includegraphics[scale=.7]{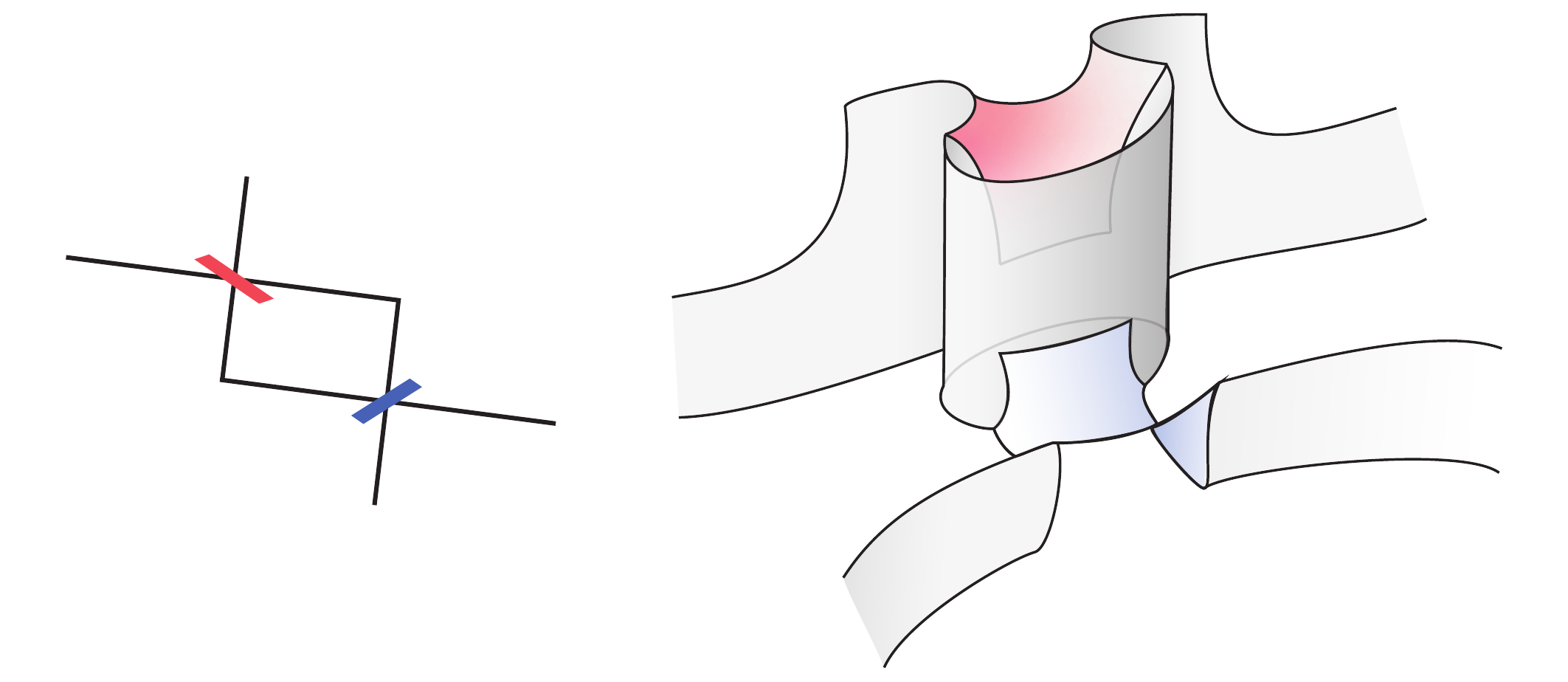}
\caption{Ribbon handle associated with the given marked vertex pattern.}
\label{fig:handle}
\end{figure}

 As a corollary to Theorem \ref{thm:ribbon}, all nontrivial surface-links of Yoshikawa's table except $8_1^{-1,-1}$, $10_2$, $10_3$, $10_1^{1,1}$, $10_2^{0,-2}$, and $10_1^{-1,-1}$ are ribbon.

\begin{theorem}[Satoh '01 \cite{satoh3}, Kamada and Oshiro '09 \cite{sym}]

The triple point number of $8_1^{-1,-1}$ is 2.

\end{theorem}

\noindent A surface in $\mathbb{R}^4$ is called ${\bf P}^2$-{\it irreducible} if it is not the connected sum $F_1\# F_2$ where $F_1$ is any surface and $F_2$ is one of the two standard projective planes, i.e. $2_1^{-1}$ or its mirror. All surfaces in Yoshikawa's table except for $2_1^{-1}$ are irreducible, Section 5 of \cite{Yoshikawa}. Satoh's lower bound calculation relies on each component being non-orientable, ${\bf P}^2$-irreducible, and having nonzero normal Euler number. 
 
 \begin{lemma}[Satoh '01 \cite{satoh3}]
 
 For a ${\bf P}^2$-irreducible surface-link $F=F_1\cup\cdots \cup F_n$ \[ t(F)\geq (|e(F_1)|+\cdots+|e(F_n)|)/2,\] where $e(F_i)$ denotes the normal Euler number of the surface-knot $F_i$. 
 
 \label{lem}
 \end{lemma}
 
  \noindent Kamada and Oshiro later proved $t(8_1^{-1,-1})=2$  using the symmetric quandle cocycle invariant. This method does not depend on the normal Euler number of the surface-link's components.

\begin{theorem}[Satoh and Shima '04 \cite{satoh1} '05 \cite{satoh6}]

The triple point number of the 2-twist-spun trefoil is 4, and the triple point number of the 3-twist-spun trefoil is 6.

\end{theorem}

\noindent Since $10_2$ represents the 2-twist-spun trefoil and $10_3$ represents the 3-twist-spun trefoil, Satoh and Shima's results give $t(10_2)=4$ and $t(10_3)=6$.

%\begin{theorem}[Theorem 5.6.4 of \cite{kamada2017surface}]
%
%A surface-knot $F$ is ribbon if and only if one can deform it by an ambient isotopy so that $F$ satisfies:
%
%\begin{itemize}
%\item[(i)] $F$ is in a $KSS$-normal form.
%
%\item[(ii)] $F$ is symmetric with respect 
%\end{itemize}
%
%\end{theorem}

\begin{theorem}

The surface-link $10_1^{1,1}$ is pseudo-ribbon.

\end{theorem}

\begin{proof}

Figure \ref{fig:10_11} shows that both the negative and positive resolutions of $10_1^{1,1}$'s marked vertex diagram only need Reidemeister II moves to achieve crossing-less diagrams. Thus, there is a broken sheet diagram representing $10_1^{1,1}$ with no triple points.

\end{proof}

\begin{figure}[h]

\begin{overpic}[unit=.5mm,scale=.6]{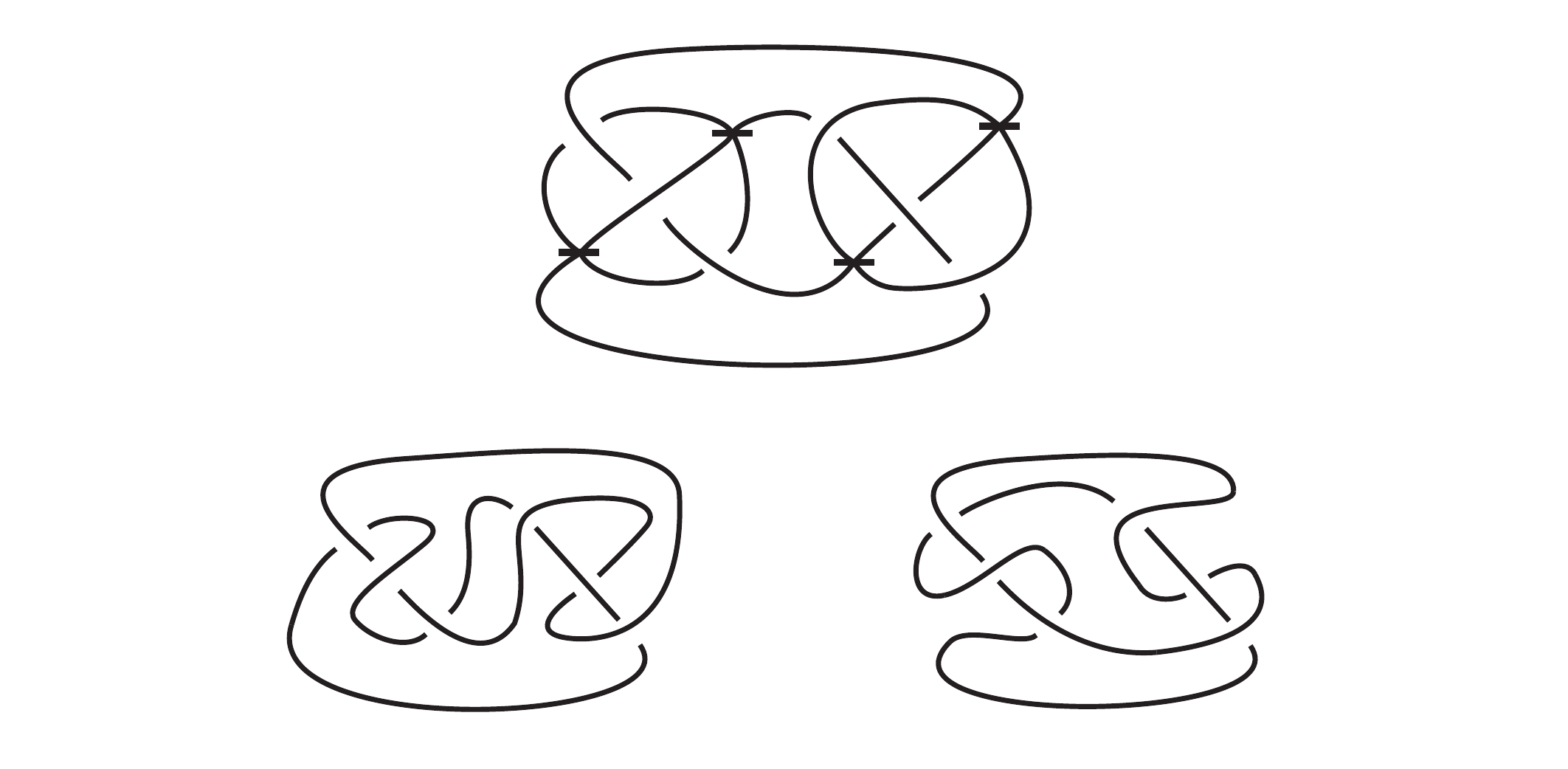}

\end{overpic}

\caption{Negative and positive resolutions of $10_1^{1,1}$.}
\label{fig:10_11}
\end{figure}

\begin{theorem}

The surface-link $10_1^{-1,-1}$ has a triple point number no greater than 12 and no less than 2.

\end{theorem}

\begin{proof}
Figures \ref{fig:neg} and \ref{fig:pos} illustrate a motion picture of $10_1^{-1,-1}$ whose induced broken sheet diagram has 12 triple points. Since $10_1^{-1,-1}$ is ${\bf P}^2$-irreducible \cite{Yoshikawa} and each component is a projective plane, Lemma \ref{lem} implies that $2\leq t(10_1^{-1,-1}).$

\end{proof}

\begin{table}[ht]
\setlength{\tabcolsep}{10pt} % Default value: 6pt
\renewcommand{\arraystretch}{1.5}
\centering
\begin{tabular}{|c||c |c |c |c| c| c |c |c |c| c| c| c| c| c| c| c| c| c| c| c| c| }

\hline 
$F$ & $0_1$ & $2_1^1$ & $2_1^{-1}$ & $6_1^{0,1}$ & $7_1^{0,-2}$ & $8_1$ & $8_1^{1,1}$ & $8_1^{-1,-1}$ & $9_1$ & $9_1^{0,1}$& $9_1^{1,-2}$    \\ \hline

$t(F)$ & 0& 0& 0& 0& 0& 0& 0& 2& 0& 0& 0 \\\hline

\end{tabular}

\vspace{5mm}

\begin{tabular}{|c||c |c| c| c| c| c| c| c| c| c| c| c| c| c| c| c| c| c| c| c| c| c| }

\hline
$F$ & $10_1$ & $10_2$ & $10_3$ & $10_1^1$ & $10_1^{0,1}$ & $10_2^{0,1}$ & $10_1^{1,1}$ & $10_1^{0,0,1}$  \\\hline 
$t(F)$ & 0& 4 & $6$& 0 &0 &0&0&0\\ \hline

\end{tabular}

\vspace{5mm}

\begin{tabular}{|c||c |c| c| c| c| c| c| c| c| c| c| c| c| c| c| c| c| c| c| c| c| c| }

\hline
$F$  & $10_1^{0,-2}$ & $10_2^{0,-2}$ & $10_1^{-1,-1}$ & $10_1^{-2,-2}$  \\\hline 
$t(F)$  & 0 &  $ t(10_2^{0,-2}) \leq 10$ & $2\leq t(10_1^{-1,-1})\leq 12$ & 0\\ \hline

\end{tabular}

\vspace{5mm}
\caption{Triple point number data on the surface-links represented in Yoshikawa's table \cite{Yoshikawa}.}
\label{tab:1}
\end{table}

\begin{theorem}

The surface-link $10_2^{0,-2}$ has a triple point number no greater than 10.
\end{theorem}

\begin{proof}

A motion picture of an induced broken sheet diagram of $10_2^{0,-2}$ is shown in Figure \ref{fig:100-2}. This motion picture has 10 Reidemeister III moves that occur between the stills connected by an arrow.

\end{proof}

\begin{remark}
Figure \ref{fig:100-2} shows that the Klein bottle component of $10_2^{0,-2}$ has a normal Euler number of 0. Therefore, Lemma \ref{lem} does not give a positive lower bound. The symmetric quandle cocycle invariant has proven useful for generating lower bounds on the triple point number of non-orientable surface-links, see \cite{carter2009symmetric}, \cite{sym}, \cite{oshiro}, and \cite{Oshiro2011}. A symmetric quandle must have a good involution with a fixed point in order to color $10_2^{0,-2}$. The trivial symmetric quandles of  \cite{Oshiro2011} color $10_2^{0,-2}$ but the given symmetric quandle cocycles have a weight of 0 for all such coloring. The dihedral quandle $R_4$ colors  $10_2^{0,-2}$ and the identity is a good involution of $R_4$, but there are no calculated symmetric quandle cocycles of $(R_4, \text{id})$. None of the other explicitly defined symmetric quandles of \cite{carter2009symmetric}, \cite{sym}, \cite{oshiro}, and \cite{Oshiro2011} color $10_2^{0,-2}$ or admit a coloring that gives a non-zero weight with the explicitly defined symmetric quandle cocycles.

\end{remark}

\section*{Acknowledgements} I would like to thank my advisor Jennifer Schultens for her revisions and encouragement.

\begin{figure}[h]
 
\begin{overpic}[unit=.399mm,scale=.6]{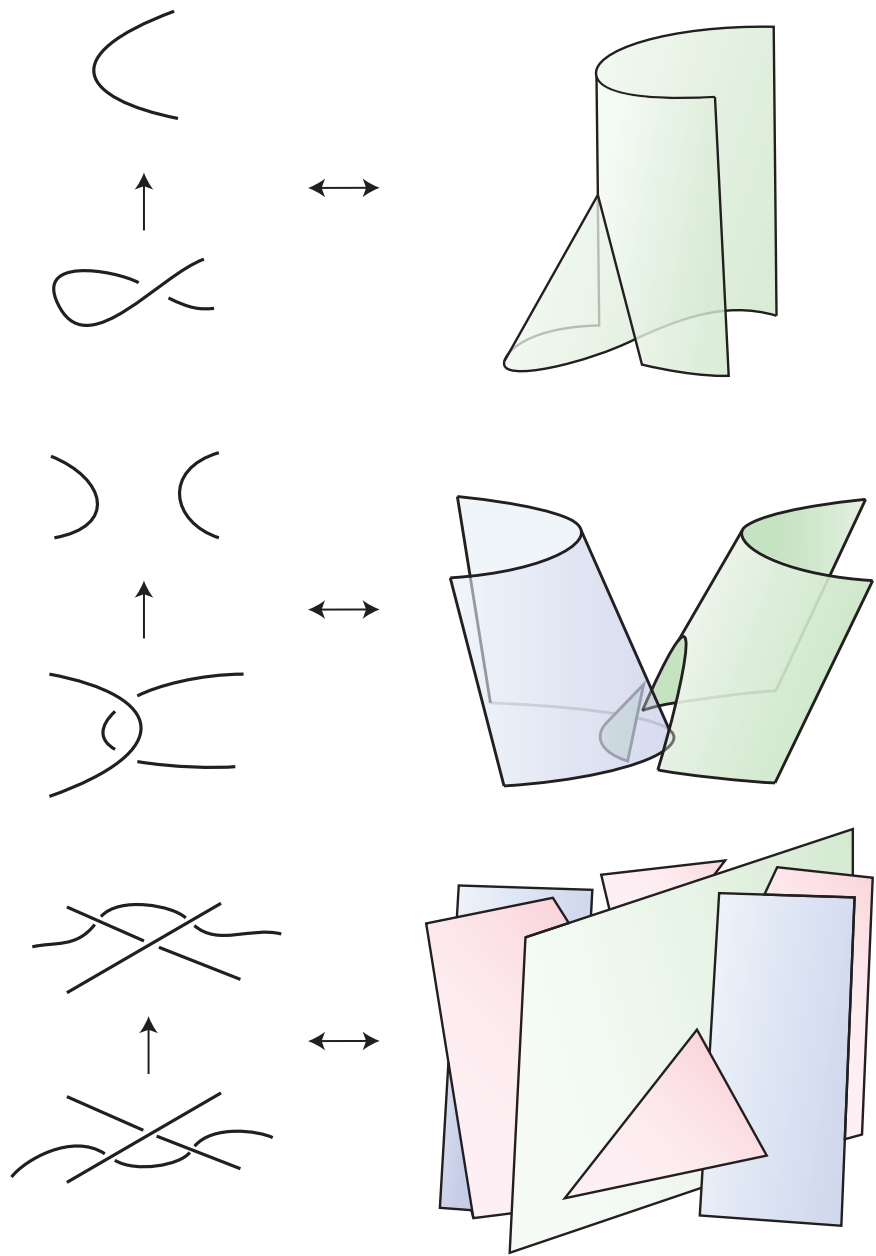}

\end{overpic}

\caption{Relationship between Reidemeister moves in a motion picture and broken sheet diagrams.}
\label{fig:r}

 \end{figure}

\begin{figure}[h]

\begin{overpic}[unit=.5mm,scale=.75]{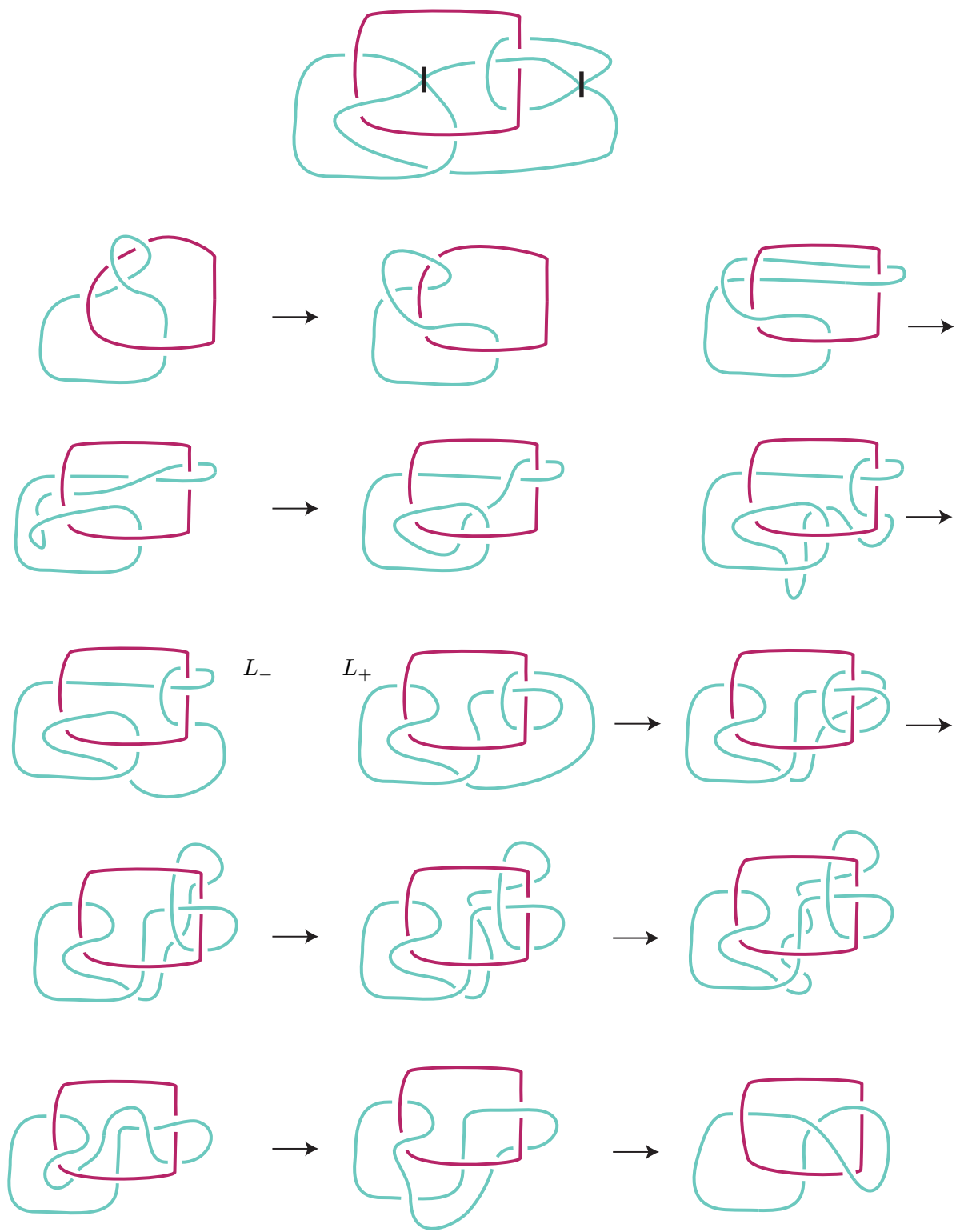}\put(84,190){$L_-$}\put(117,190){$L_+$}

\end{overpic}

\caption{A motion picture of $10_2^{0,-2}$.}
\label{fig:100-2}
\end{figure}

\begin{figure}[h]

\begin{overpic}[unit=.5mm,scale=.75]{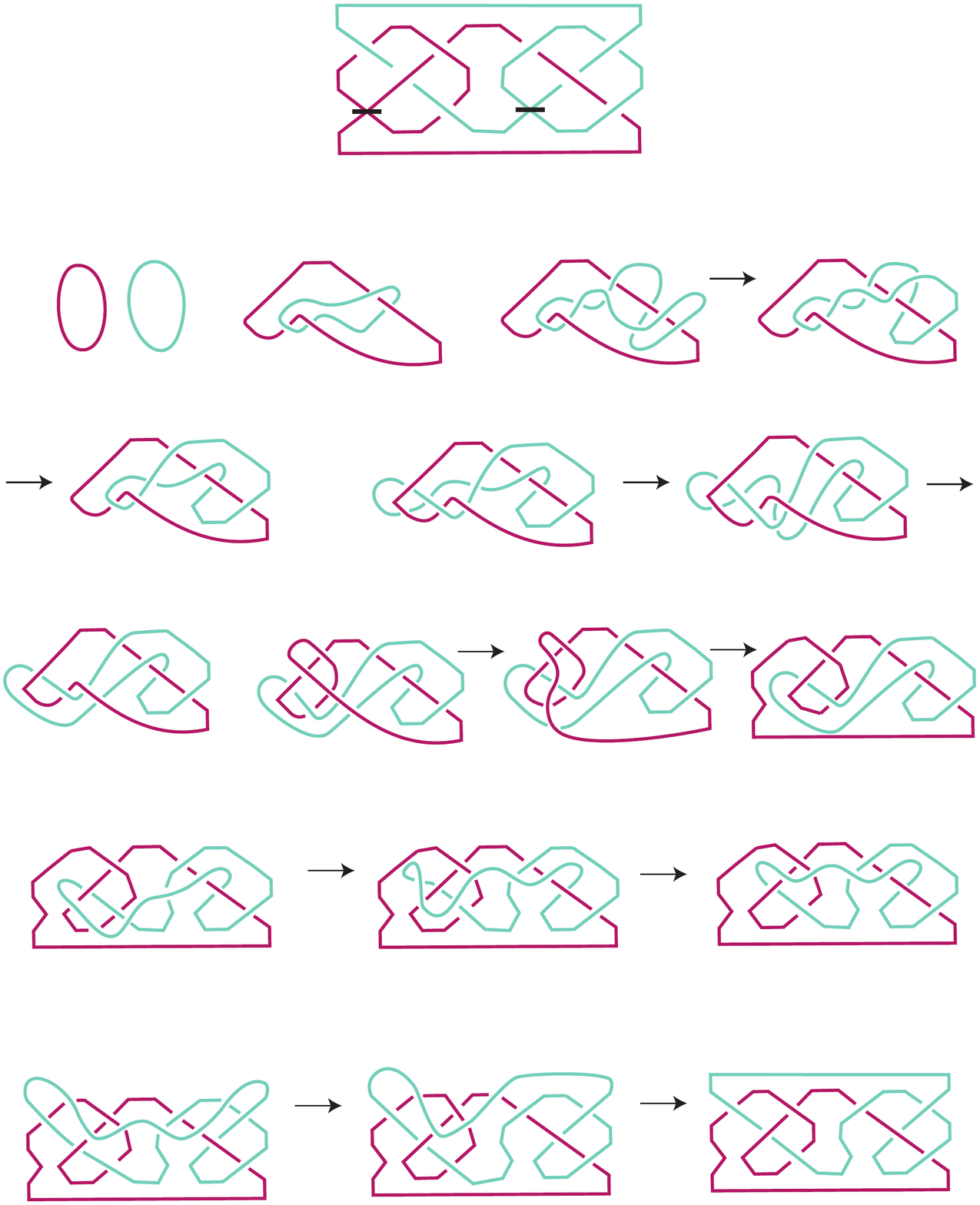}\put(318,54){$L_-$}

\end{overpic}

\caption{A motion picture of $10_1^{-1,-1}$, 1 of 2.}
\label{fig:neg}
\end{figure}

\begin{figure}[h]

\begin{overpic}[unit=.5mm,scale=.75]{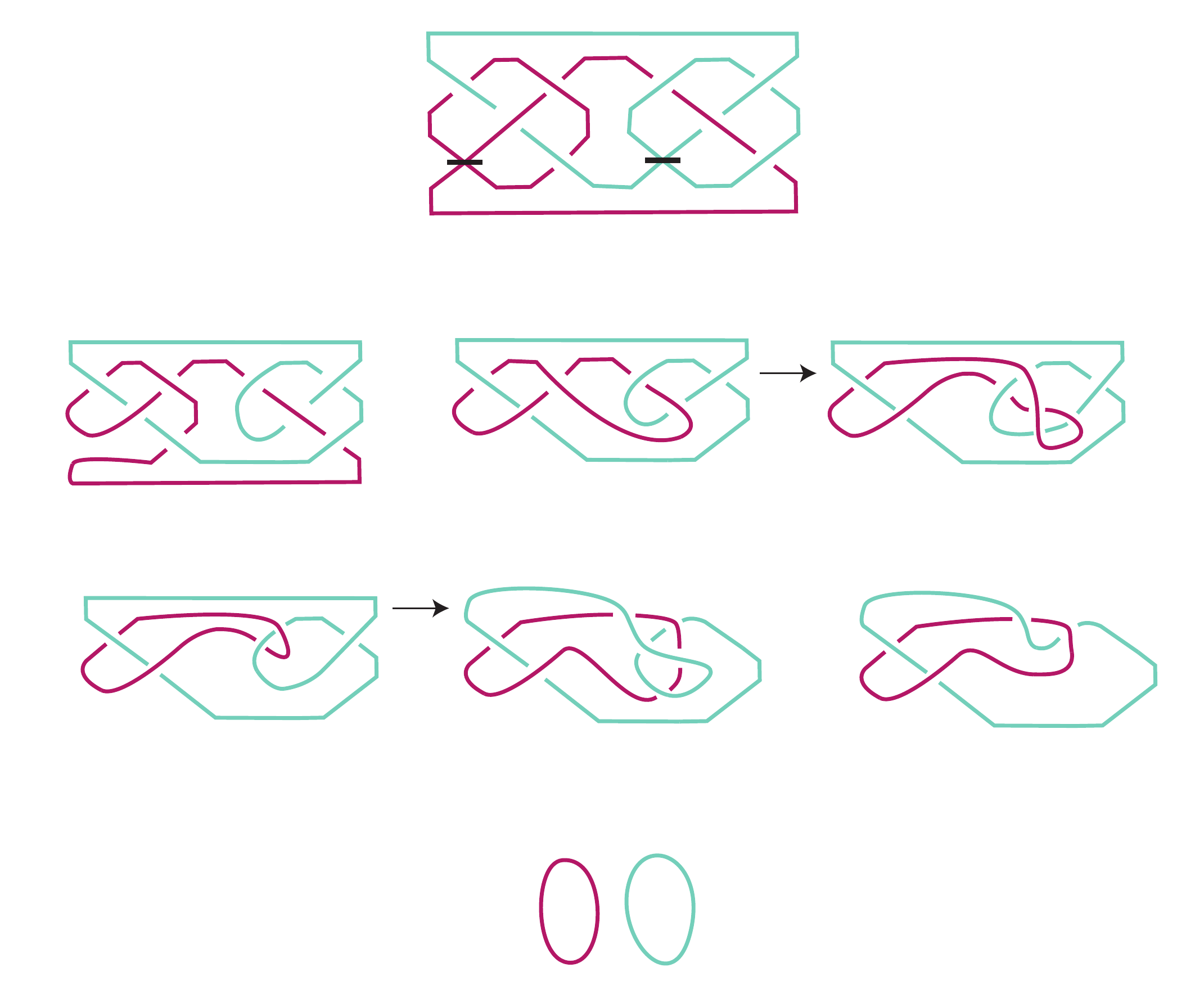}\put(2,184){$L_+$}

\end{overpic}

\caption{A motion picture of $10_1^{-1,-1}$, 2 of 2.}
\label{fig:pos}
\end{figure}

        \bibliographystyle{amsplain}
            \bibliography{proposal}

\end{document}